\RequirePackage{fix-cm}
\documentclass[smallextended]{svjour3}       
\smartqed  
\usepackage{graphicx}

 \newtheorem{cor}[theorem]{Corollary}
 {\rm}

 {\rm}
 
 \newtheorem{rem}[theorem]{Remark}
 \newtheorem{ex}{Example}

\usepackage{amsfonts,amsmath}

\def\x{\mathbf{x}}
\def\B{\mathbf{B}}

\def\R{\mathbb{R}}

\def\N{\mathbb{N}}
\def\D{\mathbf{D}}

\def\U{\mathbf{U}}
\def\K{\mathbf{K}}

\def\M{\mathbf{M}}

\def\P{\mathbf{P}}
\def\Z{\mathbf{Z}}
\def\H{\mathbf{H}}
\def\A{\mathbf{A}}
\def\bF{\mathbf{F}}

\def\B{\mathbf{B}}

\def\D{\mathbf{D}}
\def\bS{\mathbf{S}}
\def\bR{\mathbf{R}}

\def\Y{\mathbf{Y}}

\def\f{\mathbf{f}}
\def\z{\mathbf{z}}

\def\y{\mathbf{y}}

\def\y{\mathbf{y}}

\def\w{\mathbf{w}}

\def\u{\mathbf{u}}
\def\p{\mathbf{p}}

\def\s{\mathcal{S}}

\def\vol{{\rm vol}\,}
\def\dis{\displaystyle}

\begin{document}

\title{Tractable approximations of sets defined with quantifiers\thanks{The author wishes to acknowledge
financial support from  the Isaac Newton Institute for the Mathematical Sciences (Cambridge, UK) 
where this work was completed while the author was Visiting Fellow in July 2013. This work was also supported from 
a grant from the (french) Gaspar Monge Program for Optimization and Operations Research (PGMO). }}


\titlerunning{Approximations of sets defined with quantifiers}        

\author{Jean B. Lasserre}

\institute{Jean B. Lasserre \at
              LAAS-CNRS and Institute of Mathematics,
             University of Toulouse. LAAS, 7 Avenue du Colonel Roche, BP 54200, 31031 Toulouse Cedex 4, France.\\
              Tel.: +33561336415\\
              Fax: +33561336936\\
              \email{lasserre@laas.fr}}

\date{Received: date / Accepted: date}

\maketitle

\begin{abstract}
Given a compact basic semi-algebraic set $\K\subset\R^n\times\R^m$, 
a simple set $\B$ (box 
or ellipsoid), and some semi-algebraic function $f$, we consider sets defined with quantifiers, of the form
\begin{eqnarray*}
\bR_f:=\{\x\in\B&:& \mbox{$f(\x,\y)\leq 0$ for all $\y$ such that $(\x,\y)\in\K$}\}\\
\D_f:=\{\x\in\B&:& \mbox{$f(\x,\y)\geq 0$ for some $\y$ such that $(\x,\y)\in\K$}\}.\end{eqnarray*}
The former set $\bR_f$ is particularly useful to qualify
``robust" decisions $\x$ versus noise parameter $\y$ (e.g. in robust optimization on some set $\mathbf{\Omega}\subset\B$) whereas
the latter set $\D_f$ is useful (e.g. in optimization) when one does not want to work with 
its lifted representation $\{(\x,\y)\in\K: f(\x,\y)\geq 0\}$. Assuming that
$\K_\x:=\{\y:(\x,\y)\in\K\}\neq\emptyset$ for every $\x\in\B$, we provide a 
systematic procedure to obtain a sequence of explicit inner (resp. outer) approximations
that converge to $\bR_f$ (resp. $\D_f$) in a strong sense. 
Another (and remarkable) feature is that each approximation
is the sublevel set of a single polynomial whose vector of coefficients is an optimal solution of a semidefinite program.
Several extensions are also proposed, and in particular, approximations for sets of the form
\[\bR_F:=\{\x\in\B\::\: \mbox{$(\x,\y)\in F$ for all $\y$ such that $(\x,\y)\in\K$}\}\]
where $F$ is some other basic-semi algebraic set, and also sets defined with two quantifiers.
\end{abstract}

\section{Introduction}

Consider two sets of variables $\x\in\R^n$ and $\y\in\R^m$ 
coupled with a constraint $(\x,\y)\in\K$, where $\K\subset\R^n\times\R^m$ is some compact basic semi-algebraic set\footnote{A basic semi-algebraic set is the intersection $\cap_{j=1}^m\{\x:g_j(\x)\geq0\}$ of super level sets 
of finitely many polynomials $(g_j)\subset\R[\x]$.} 
defined by:
\begin{equation}
\label{basic-K}
\K\,:=\,\{(\x,\y)\in\R^n\times\R^m\::\:\x\in\B;\quad g_j(\x,\y)\,\geq\,0,\:j=1,\ldots,s\}
\end{equation}
for some polynomials $g_j$, $j=1,\ldots,s$, and 
let $\B\subset\R^n$ be a simple set (e.g. some box or ellipsoid).

With $f:\K\to\R$ a given semi-algebraic function on $\K$ (that is, a function whose graph $\Psi_f:=\{(\x,f(\x)):\x\in\K\}$ is a semi-algebraic set),
and 
\begin{equation}
\label{setKx}
\K_\x\,:=\,\{\y\,\in\,\R^m\::\:(\x,\y)\,\in\,\K\:\},
\end{equation}
consider the two sets:

\begin{equation}
\label{robust-set}
\bR_f\,:=\,\{\:\x\in\B\::\:f(\x,\y)\leq 0 \mbox{ for all $\y\in\K_\x$}\:\},
\end{equation}
and
\begin{equation}
\label{design-set}
\D_f\,:=\,\{\:\x\in\B\::\:f(\x,\y)\geq 0 \mbox{ for some $\y\in\K_\x$}\:\}.
\end{equation}
Both sets $\bR_f$ and $\D_f$ which include a {\it quantifier} in their definition, are semi-algebraic and are interpreted as {\it robust} sets of variables $\x$
with respect to the other set of variables $\y$, and to some performance criterion $f$. \\

Indeed in the first case (\ref{robust-set}) one may think of ``$\x$" as {\it decision} variables which should be {\it robust}
with respect to some {\it noise} (or perturbation) $\y$ in the sense that
no matter what the admissible level of noise $\y\in\K_\x$ is, the constraint 
$f(\x,\y)\leq 0$ is satisfied whenever $\x\in\bR_f$. For instance, such sets $\bR_f$ are fundamental in robust control
and robust optimization on a set $\mathbf{\Omega}\subset\B$ (in which case one 
is interested in $\bR_f\cap\mathbf{\Omega}$).
Instead of considering $\mathbf{\Omega}$ directly in the definition (\ref{robust-set})
of $\bR_f$ one introduces the simple set 
$\B\supset\mathbf{\Omega}$ because {\it moments} 
of the Lebesgue measure on $\B$ (which are needed later) are easy to compute 
(in contrast to moments of the Lebesgue measure on $\mathbf{\Omega}$). For a nice treatment
of robust optimization the interested reader is referred to Ben-Tal et al. \cite{bental} where a particular 
emphasis is put on how to model the uncertainty so as to obtain {\it tractable} formulations 
for robust counterparts of some (convex) conic optimization problems. We are here interested in (converging) approximations
of sets $\bR_f$ in the general framework of polynomials.

On the other hand, in the second case (\ref{design-set})
the vector $\x$ should be interpreted as {\it design} variables (or parameters) and
the set $\K_\x$ defines a set of admissible decisions $\y\in\K_\x$
within the framework of design $\x$.
And so $\D_f$ is the set of {\it robust} design parameters $\x$, in the sense that
for every value of the design parameter $\x\in\D_f$, there is at least one admissible decision $\y\in\K_\x$
with performance level $f(\x,\y)\geq 0$. Notice that
$\D_{f}\supseteq\overline{\B\setminus\bR_f}$, and in a sense robust optimization on $\bR_f$
is dual to optimization on $\D_f$.\\

The semi-algebraic function $f$ as well as the set $\K$ can be fairly complicated and
therefore in general both sets $\bR_f$ and $\D_f$ are non convex so that their exact description
can be fairly complicated as well!
Needless to say that robust optimization problems 
with constraints of the form $\x\in\bR_f$, are very difficult solve. 
In principle when $\K$ is a basic semi-algebraic set,
quantifier elimination is possible via algebraic techniques; see e.g. Bochnak et al. \cite{bochnak}.
However, in practice quantifier elimination is very costly and intractable.

On the other hand, optimization problems with a constraint of the form $\x\in\D_f$
(or $\x\in\D_f\cap\mathbf{\Omega}$ for some $\mathbf{\Omega}$)
can be formulated directly in the lifted space 
of variables $(\x,\y)\in\R^n\times\R^m$ (i.e. by adding the constraints $f(\x,\y)\geq0;\:(\x,\y)\in\K$)
and so with no approximation. But sometimes one may be interested in
getting a description of the set $\D_f$ itself in $\R^n$ because its ``shape" is
hidden in the lifted $(\x,\y)$-description, or because optimizing over $\K\cap \{(\x,\y):f(\x,\y)\geq0\}$
may not be practical. However, if the projection 
of a basic semi-algebraic set (like e.g. $\D_f$) is semi-algebraic, it is not necessarily {\it basic} semi-algebraic and could be 
a complicated union of several basic semi-algebraic sets (hence not very useful in practice). 

So a less ambitious but more practical goal is
to obtain {\it tractable} approximations of such sets $\bR_f$ (or $\D_f$).
Then such approximations can be used for various purposes, optimization being only one potential
application.

\paragraph{Contribution.}
In this paper we provide a hierarchy $(\bR^k_f)$ (resp. $(\D_f^k)$), $k\in\N$, of {\it inner} approximations  
for $\bR_f$ (resp. {\it outer} approximations for $\D_f$). These two hierarchies have three essential characteristic features:

- (a) Each set $\bR_f^k\subset\R^n$ (resp. $\D_f^k$), $k\in\N$, has a very simple description
in terms of the sublevel set $\{\,\x\in\B: p_k(\x)\leq0\,\}$ (resp. $\{\,\x\in\B: p_k(\x)\geq0\,\}$)
associated with a {\it single} polynomial $p_k$.

- (b) Both hierarchies $(\bR_f^k)$ and $(\D_f^k)$, $k\in\N$, converge in a strong sense since we prove
that  (under some conditions) ${\rm vol}\,(\bR_f\setminus \bR^k_f)\to 0$  (resp.  ${\rm vol}\,(\D_f^k\setminus \D_f)\to 0$)
as $k\to\infty$ (and where ``${\rm vol}(\cdot)$" denotes the Lebesgue volume). In other words, 
for $k$ sufficiently large the inner approximations $\bR_f^k$ (resp. outer approximations $\D_f^k$) coincide with $\bR_f$ (resp. $\D_f$) up to a set 
of very small Lebesgue volume.

- (c) Computing the vector of coefficients of the above
polynomial $p_k$ reduces to solving a semidefinite program whose size is parametrized by $k$.

As one potential application, the constraint $p_k(\x)\leq0$ (resp. $p_k(\x)\geq0$) 
can be used in any robust (resp. design) optimization problem on $\mathbf{\Omega}\,\subseteq\B$ 
as a substitute for $\x\in\bR_f\cap\,\mathbf{\Omega}$ (resp. $\x\in\D_f\cap\,\mathbf{\Omega}$), thereby eliminating 
the variables $\y$. 
One then obtains a standard polynomial
minimization problem $\P$ for which one may apply the hierarchy of semidefinite relaxations defined in \cite{lass-book}
to obtain a sequence of lower bounds on the optimal value (and sometimes an optimal solution if the size of the resulting is moderate or
if some sparsity pattern can be used for larger size problems). For more details,
the interested reader is referred to \cite{lass-book} (and Waki et al \cite{waki} for semidefinite relaxations 
that use a sparsity pattern). This approach was proposed in \cite{lassjogo} for robust optimization 
(and in \cite{anjos-j+m} with some convergence guarantee). 
But the sets $\bR_f^k$ can also be used in other applications to provide a certificate for robustness
as membership in $\bR_f^k$ is easy to check and the approximation is from inside.

We first obtain inner (resp. outer) approximations of $\bR_f$ (resp. $\D_f$) when $f$ is a {\it polynomial}. 
To do so we extensively use a previous result of the author \cite{lass-param} 
which allows to approximate in a strong sense the optimal value of a parametric optimization problem.
We then describe how the methodology can be extended to the case where $f$ is a semi-algebraic function on $\K$, whose graph $\Psi_f$
is explicitly described by a basic semi-algebraic set\footnote{That is,
$\Psi_f=\{(\x,\y,f(\x,\y)):(\x,\y)\,\in\,\K\}=\{(\x,\y,v):(\x,\y)\in\K;\:\exists\w\mbox{ s.t. }h_\ell(\x,\y,\w,v)\geq0,\,\ell=1,\ldots,s\}$,
for some polynomials $h_\ell\in\R[\x,\y,\w,v]$.}.
This methodology had been already used in Henrion and Lasserre \cite{lass-ieee} to provide (convergent) inner approximations 
for the particular case of a set defined by matrix polynomial inequalities. 
The present contribution can be viewed as an extension
of \cite{lass-ieee} to the more general framework (\ref{robust-set})-(\ref{design-set}) and with $f$ semi-algebraic.

Finally, we also provide several extensions, and in particular, we consider:

- The case where one also enforces the computed inner or outer approximations to be
a {\it convex} set. This can be interesting for optimization purposes but of course, in this case convergence as in (b) is lost.


- The case where $f(\x,\y)\leq0$ is now replaced with a polynomial matrix inequality $\bF(\x,\y)\preceq0$, i.e.,
$\bF(\cdot,\cdot)$ is a real symmetric $m\times m$ matrix such that $\bF_{ij}\in\R[\x,\y]$ for each entry $(i,j)$.
One then retrieves the methodology already used in Henrion and Lasserre \cite{lass-ieee} to provide (convergent) inner approximations 
of the set $\{\x\in\B: \bF(\x)\succeq0\}$ for some polynomial matrix inequality.

- The case where $\bR_f$ is now replaced with the set $\bR_F$ defined by:
\[\bR_F\,=\,\{\x\in\B\::\: \mbox{$(\x,\y)\in F$ for all $\y$ such that $(\x,\y)\in\K$}\},\]
where $F$ is some basic-semi-algebraic set. And a similarly extension is also
possible for sets $\D_f$ defined accordingly.

- The case where we now how have two quantifiers, like for instance,
\[\bR_f\,=\,\{\x\in\B_\x\::\:\exists\,\y\in\B_\y\mbox{ s.t. }f(\x,\y,\u)\leq0,\:\forall \u:\:(\x,\y,\u)\in\K\},\]
for some boxes $\B_\x\subset\R^n$, $\B_\y\subset\R^m$, and some compact set $\K\subset\R^n\times\R^m\times\R^s$.

- The case where $\K$ is a semi-algebraic (but not basic semi-algebraic) set.

\section{Notation and definitions}
\label{notation}

In this paper we  use some material which is now more or less standard in what is called
{\it Polynomial Optimization}. The reader not totally familiar with such notions and 
concepts, will find suitable additional material and details in the sources \cite{sos}, \cite{lass-book} and \cite{Laurent}.

Let $\R[\x]$ denote the ring or real polynomials in the variables $\x=(x_1,\ldots,x_n)$, and 
let $\R[\x]_d$ be the vector space of real polynomials of degree at most $d$.
Similarly, let $\Sigma[\x]\subset\R[\x]$ denote the convex cone of real polynomials that are sums of squares (SOS) of polynomials,
and $\Sigma[\x]_d\subset\Sigma[\x]$ its subcone of SOS polynomials of degree at most $2d$.
Denote by $\s^m$ the space of $m\times m$ real symmetric matrices. For a given matrix $\A\in\s^m$, the notation
$\A \succeq 0$ (resp. $\A\succ0$) means that $\A$ is positive semidefinite (resp. positive definite), i.e., all its eigenvalues 
are real and nonnegative (resp. positive).  For a Borel set $B\subset\R^n$ let
${\rm vol}(B)$ denote its Lebesgue volume.

\paragraph{Moment matrix.} With $\z=(z_\alpha)$ being a sequence indexed in the canonical basis
$(\x^\alpha)$ of $\R[\x]$, let $L_\z:\R[\x]\to\R$ be the so-called Riesz functional defined by:
\[f\quad (=\sum_{\alpha}f_{\alpha}\,\x^\alpha)\quad\mapsto\quad
L_\z(f)\,=\,\sum_{\alpha}f_{\alpha}\,z_{\alpha},\]
and let $\M_d(\z)$ be the symmetric matrix with rows and columns indexed in 
the canonical basis $(\x^\alpha)$, and defined by:
\begin{equation}
\label{moment}\M_d(\z)(\alpha,\beta)\,:=\,L_\z(\x^{\alpha+\beta})\,=\,z_{\alpha+\beta},\quad\alpha,\beta\in\N^n_d\end{equation}
with $\N^n_d:=\{\alpha\in\N^n\::\:\vert \alpha\vert \:(=\sum_i\alpha_i)\leq d\}$.

If $\z$ has a representing measure $\mu$, i.e., if
$z_\alpha=\int \x^\alpha d\mu$ for every $\alpha\in\N^n$, then
\[\langle \f,\M_d(\z)\f\rangle\,=\,\int f(\x)^2\,d\mu(\x)\,\geq\,0,\qquad \forall \,f\in\R[\x]_d,\]
and so $\M_d(\z)\succeq0$. In particular, if $\mu$ has a density $h$ with respect
to the Lebesgue measure, positive on some open set $B$, then $\M_d(\z)\succ0$ because 
\[0\,=\,\langle \f,\M_d(\z)\f\rangle\,\geq\,\int_B f(\x)^2\,h(\x)d\x\quad\Rightarrow\quad f=0.\]

\paragraph{Localizing matrix.} Similarly, with $\z=(z_{\alpha})$
and $g\in\R[\x]$ written
\[\x\mapsto g(\x)\,=\,\sum_{\gamma\in\N^n}g_{\gamma}\,\x^\gamma,\]
let $\M_d(g\,\y)$ be the symmetric matrix with rows and columns indexed in 
the canonical basis $(\x^\alpha)$, and defined by:
\begin{equation}
\label{localizing}
\M_d(g\,\z)(\alpha,\beta)\,:=\,L_\z\left(g(\x)\,\x^{\alpha+\beta}\right)\,=\,\sum_{\gamma}g_{\gamma}\,
z_{\alpha+\beta+\gamma},\qquad\forall\,
\alpha,\beta\in\N^n_d.\end{equation}
If $\z$ has a representing measure $\mu$,
then $\langle \f,\M_d(g\,\z)\f\rangle=\int f^2gd\mu$,
and so if $\mu$ is supported on the set $\{\x\,:\,g(\x)\geq0\}$,
then $\M_d(g\,\z)\succeq0$ for all $d=0,1,\ldots$ because
\begin{equation}
\label{localizing2}
\langle \f,\M_d(g\,\z)\f\rangle\,=\,\int f(\x)^2g(\x)\,d\mu(\x)\,\geq\,0,\qquad \forall \,f\in\R[\x]_d.\end{equation}
In particular, if $\mu$ is the Lebesgue measure and $g$ is positive on some open set $B$, then $\M_d(g\,\z)\succ0$ because 
\[0\,=\,\langle \f,\M_d(g\,\z)\f\rangle\,\geq\,\int_B f(\x)^2\,g(\x)d\x\quad\Rightarrow\quad f=0.\]

\section{Main result}
\label{main}

Let $\K$ be the basic semi-algebraic set defined in (\ref{basic-K}) 
for some polynomials $g_j\subset\R[\x,\y]$, $j=1,\ldots,s$, and with 
simple set (box or ellipsoid) $\B\subset\R^n$.

Denote by $L_1(\B)$ the Lebesgue space of measurable functions $h:\B\to\R$ that are integrable with respect to the Lebesgue measure on $\B$, i.e., such that $\int_\B\vert h\vert d\x<\infty$.

Given $f\in\R[\x,\y]$, consider the mapping $\overline{J_f}:\B\to\R\cup\{-\infty\}$ defined by:
\begin{equation}
\label{upper}
\x\mapsto \overline{J_f}(\x)\,:=\,\dis\max_{\y}\,\{\,f(\x,\y):\y\in\,\K_\x\,\},\qquad\x\,\in\,\B.\end{equation}
Therefore the set $\bR_f$ in (\ref{robust-set}) reads $\{\,\x\in\B: \overline{J_f}(\x)\leq0\,\}$ whereas
$\D_f$ in (\ref{design-set}) reads $\{\x\in\B: \overline{J_f}(\x)\geq0\,\}$.
\begin{lemma}
The function $\overline{J_f}$ is upper semi-continuous.
\end{lemma}
\begin{proof}
With $\x\in\B$ let $(\x_n)\subset\B$, $n\in\N$, be a sequence such that $\x_n\to\x$ and
$\limsup_{\z\to\x}\overline{J_f}(\z)=\lim_{n\to\infty}\overline{J_f}(\x_n)$. As $\K$ is compact, for every $n\in\N$,
$\overline{J_f}(\x_n)=f(\x_n,\y_n)$ for some $\y_n\in\K_{\x_n}$. Therefore there is some subsequence
$(n_k)$, $k\in\N$, and some $\y$ with $(\x,\y)\in\K$, such that$(\x_{n_k},\y_{n_k})\to(\x,\y)\in\K$ as $k\to\infty$. Hence
\begin{eqnarray*}
\limsup_{\z\to\x}\overline{J_f}(\z)&=&\lim_{k\to\infty}\overline{J_f}(\x_{n_k})\,=\,\lim_{k\to\infty}f(\x_{n_k},\y_{n_k})\\
&=&f(\x,\y)\leq\max_\z\,\{\,f(\x,\z):\z\in\K_\x\,\}\,=\,\overline{J_f}(\x),
\end{eqnarray*}
i.e., $\limsup_{\z\to\x}\overline{J_f}(\z)\leq\overline{J_j}(\x)$, the desired result. $\qed$
\end{proof}
We will need also the following intermediate result.
\begin{theorem}
\label{th-over-under}
Let $\B\subset\R^n$ be a compact set and $\overline{J_f}:\B\to\R$ a bounded and upper semi-continuous function.
Then there exists a sequence of polynomials $\{p_k:k\in\N\}\subset\R[\x]$
such that $p_k(\x)\geq \overline{J_f}(\x)$ for all $\x\in\B$ and 
\begin{equation}
\label{th-up}
\lim_{k\to\infty}\:\int_\B\vert \,p_k(\x)-\overline{J_f}(\x)\vert\,d\x\,=\,0\quad\mbox{[Convergence in $L_1(\B)$]}.\end{equation}
\end{theorem}
\begin{proof}
To  prove (\ref{up}) observe that $\overline{J_f}$ being bounded and upper semi-continuous on $\B$,
there exists a nonincreasing sequence $(f_k)$, $k\in\N$, 
of bounded continuous functions $f_k:\B\to \R$ such that
$f_k(\x)\downarrow \overline{J_f}(\x)$ for all $\x\in\B$, as $k\to\infty$; see e.g. Ash \cite[Theorem A6.6, p. 390]{ash}.
Moreover, by the Monotone Convergence Theorem:
\[\int_{\B} f_k(\x)\,d\x\:\to\:\int_{\B} \overline{J_f}(\x)\,d\x\qquad\mbox{as $k\to\infty$},\]
and so
\[\int_{\B} \vert f_k(\x)-\overline{J_f}(\x)\vert\,d\x\,=\,
\int_{\B} (f_k(\x)-\overline{J_f}(\x))\,d\x\:\to\:0\qquad\mbox{as $k\to\infty$,}\]
that is, $f_k\to\overline{J_f}$ for the $L_1(\B)$-norm.
Next, by the Stone-Weierstrass  theorem, for every $k\in\N$, there exists $p'_k\in\R[\x]$ such that
$\sup_{\x\in\B}\vert p'_k-f_k\vert<(2k)^{-1}$ and so $p_k:=p'_k+(2k)^{-1}\geq f_k\geq \overline{J_f}$ on $\B$.
In addition,
\begin{eqnarray*}
\lim_{k\to\infty}\int_{\B}\vert p_k(\x)-\overline{J_f}(\x)\vert\,d\x&=&
\lim_{k\to\infty}\int_{\B}\underbrace{\vert p_k(\x)-f_k(\x)\vert}_{\leq k^{-1}}+\vert f_k(\x)-\overline{J_f}(\x)\vert\,d\x\\
&\leq&\lim_{k\to\infty}\left(k^{-1}{\rm vol}(\B)+\int_\B\vert f_k(\x)-\overline{J_f}(\x)\vert\,d\x\right)\\
&\leq&\lim_{k\to\infty}\int_\B\vert f_k(\x)-\overline{J_f}(\x)\vert\,d\x\,=\,0.\qquad\qed\end{eqnarray*}
\end{proof}
The following result is an immediate consequence of Theorem \ref{th-over-under}.
\begin{cor}
\label{over-under}
Let $\K\subset\R^n\times\R^m$ be compact.
If $\K_\x\neq\emptyset$ for every $\x\in\B$, there exists a sequence of polynomials $\{p_k:k\in\N\}\subset\R[\x]$, 
such that $p_k(\x)\geq f(\x,\y)$ for all $\y\in\K_\x$, $\x\in\B$, and 
\begin{equation}
\label{up}
\lim_{k\to\infty}\:\int_\B\vert \,p_k(\x)-\overline{J_f}(\x)\vert\,d\x\,=\,0\quad\mbox{[Convergence in $L_1(\B)$]}.\end{equation}
\end{cor}

\subsection{Inner approximations of $\bR_f$}

Let $\K$ be as in (\ref{basic-K}) with $\K_\x$ as in (\ref{setKx}) and assume that $\B$ and $\bR_f$ in (\ref{robust-set})
have nonempty interior.

\begin{theorem}
\label{robust-approx}
Let $\K\subset\R^n\times\R^m$ in (\ref{basic-K}) be compact and
$\K_\x\neq\emptyset$ for every $\x\in\B$. Assume that
$\{\,\x\in\B: \overline{J_f}(\x)\,=\,0\,\}$ has Lebesgue measure zero,
and for every $k\in\N$, let $\bR_f^k:=\{\,\x\in\B:p_k(\x)\leq0\,\}$, where $p_k\in\R[\x]$ is as in Corollary \ref{over-under}.  
Then $\bR_f^k\subset\bR_f$ for every $k$, and
\begin{equation}
\label{vol1}
{\rm vol}\left(\bR_f\setminus \bR_f^k\right)\,\to\,0\quad\mbox{ as }k\to\infty.
\end{equation}
\end{theorem}
\begin{proof}
By Corollary \ref{over-under}, $p_k\to \overline{J_f}$ in $L_1(\B)$ as $k\to\infty$. Therefore by \cite[Theorem 2.5.1]{ash},
$p_k$ converges to $\overline{J_f}$ in measure, that is, for every $\epsilon>0$,
\begin{equation}
\label{aux}
\lim_{k\to\infty}\:{\rm vol}\left(\{\,\x\::\:\vert p_k(\x)-\overline{J_f}(\x)\vert\,\geq\,\epsilon\}\right)\,=\,0.\end{equation}

Next, as $\overline{J_f}$ is upper semi-continuous
on $\B$, the set $\{\x: \overline{J_f}(\x)<0\}$ is open and as the set $\{\x\in\B:\overline{J_f}(\x)=0\}$ has Lebesgue measure zero,
\begin{eqnarray}
\nonumber
{\rm vol}(\bR_f)\,=\,{\rm vol}\left(\{\x\in \B : \overline{J_f}(\x)<0\}\right)&=&
{\rm vol}\left(\bigcup_{\ell=1}^\infty \{\x\in \B : \overline{J_f}(\x)\leq -1/\ell\}\right)\\
\nonumber
&=&\lim_{\ell\to\infty}{\rm vol}\left(\{\x\in \B : \overline{J_f}(\x)\leq -1/\ell\}\right)\\
\label{aux4}
&=&\lim_{\ell\to\infty}{\rm vol}\left(\bR_f(\ell)\right),
\end{eqnarray}
where $\bR_f(\ell):=\{\x\in\B:\overline{J_f}(\x)\leq -1/\ell\}$. Next,
$\bR_f(\ell)\subseteq\bR_f$ for every $\ell\geq1$, and
\[{\rm vol}\,(\bR_f(\ell))\,=\,{\rm vol}\left(\bR_f(\ell)\cap\{\x: p_k(\x)>0\}\right)
+{\rm vol}\left(\bR_f(\ell)\cap\{\x: p_k(\x)\leq0\}\right).\]
Observe that by (\ref{aux}), ${\rm vol}\left(\bR_f(\ell)\cap\{\x: p_k(\x)>0\}\right)\to0$ as $k\to\infty$.
Therefore,
\begin{eqnarray}
\label{aux3}
{\rm vol}(\bR_f(\ell))&=&\lim_{k\to\infty}
{\rm vol}\left(\bR_f(\ell)\right.\cap\underbrace{\left.\{\x: p_k(\x)\leq0\}\right)}_{=\bR_f^k}\\
\nonumber
&\leq&\lim_{k\to\infty}{\rm vol}\,(\bR_f^k)\,\leq\,{\rm vol}\,(\bR_f).\end{eqnarray}
As $\bR_f^k\subset\bR_f$ for all $k$, letting $\ell\,\to\infty$ and using (\ref{aux4}) yields the desired result.
$\qed$
\end{proof}
Theorem \ref{robust-approx} states that the (potentially complicated) set $\bR_f$ 
can be approximated arbitrary well from inside by sublevel sets of polynomials.
In particular, for application in robust optimization problems where one wishes to optimize a function over some set $\mathbf{\Omega}\,\cap\bR_f$ for some $\mathbf{\Omega}\subset\B$, one may reinforce the complicated (and intractable) constraint $\x\in\bR_f\cap\,\mathbf{\Omega}$ by instead considering 
the inner approximation $\{\x\in\mathbf{\Omega}:p_k(\x)\leq0\}$ obtained with the two much simpler constraints $\x\in\mathbf{\Omega}$ and $p_k(\x)\leq0$. The resulting conservatism becomes negligible as $k$ increases. 

\subsection{Outer approximations of $\D_f$}

Let $\B$ and $\D_f$ in (\ref{design-set}) have nonempty interior.

\begin{cor}
\label{design-approx}
Let $\K\subset\R^n\times\R^m$ in (\ref{basic-K}) be compact and
$\K_\x\neq\emptyset$ for every $\x\in\B$. Assume that
$\{\,\x\in\B:\overline{J_f}(\x)\,=\,0\,\}$ has Lebesgue measure zero,
and for every $k\in\N$, let $\D_f^k:=\{\,\x\in\B :p_k(\x)\geq0\,\}$, where $p_k\in\R[\x]$ is as in Corollary \ref{over-under}.  
Then $\D_f^k\supset\D_f$ for every $k$, and
\begin{equation}
\label{vol2}
{\rm vol}(\D_f\setminus \D_f^k)\,\to\,0\quad\mbox{ as }k\to\infty.
\end{equation}
\end{cor}
\begin{proof}
The proof uses same arguments as in the proof of Theorem \ref{robust-approx}.
Indeed, $\D_f=\B\setminus\Delta_f$ with
\begin{eqnarray*}
\Delta_f&:=&\{\,\x\in\B\::\: f(\x,\y)\,<\,0\mbox{ for all $\y\in\K_\x$}\}\\
&=&\{\,\x\in\B\::\: \sup_\y\{\,f(\x,\y)\::\:\y\in\K_\x\}\,<\,0\}\\
&=&\{\,\x\in\B\::\: \overline{J_{f}}(\x)\,<\,0\},
\end{eqnarray*}
and since $\{\x\in\B\::\: \overline{J_f}(\x)\,=\,0\}$ has Lebesgue measure zero,
\[{\rm vol}(\Delta_f)={\rm vol}\left(\{\x\in\B\::\: \overline{J_{f}}(\x)\,\leq\,0\}\right).\]
Hence by Theorem \ref{robust-approx}, 
\[\lim_{k\to\infty}{\rm vol}\left(\{\x\in\B\::\:p_k(\x)\leq0\}\right)\,=\,{\rm vol}\left(\Delta_f\right),\]
which in turn implies the desired result
\begin{eqnarray*}
\lim_{k\to\infty}{\rm vol}\left(\{\x\in\B\::\:p_k(\x)\geq0\}\right)&=&
\lim_{k\to\infty}{\rm vol}\left(\{\x\in\B\::\:p_k(\x)>0\}\right)\\
&=&{\rm vol}\left(\B\setminus\Delta_f\right)\,=\,{\rm vol}\left(\D_f\right),\end{eqnarray*}
because ${\rm vol}\left(\{\x\in\B:p_k(\x)=0\}\right)=0$ for every $k$.
$\quad\qed$
\end{proof}
Corollary \ref{design-approx} states that the set $\D_f$ 
can be approximated arbitrary well from outside by sublevel sets of polynomials.
In particular, if $\mathbf{\Omega}\subset\B$ and one wishes to work with
$\D_f\cap\,\mathbf{\Omega}$ and not its lifted representation $\{\,f(\x,\y)\geq0;\,\x\in\mathbf{\Omega};\:(\x,\y)\in\K\,\}$, one may instead
use the outer approximation $\{\x\in\mathbf{\Omega}:p_k(\x)\geq0\}$. 
The resulting laxism  becomes negligible as $k$ increases.

\subsection{Practical computation}

In this section we follow \cite{lass-param} and show how to compute  a sequence of polynomials $(p_k)\subset\R[\x]$, $k\in\N$,
as defined in Theorem \ref{robust-approx}. 
With $\K\subset\R^n\times\R^m$ as in (\ref{basic-K}) and compact, 
we assume that we know some $M>0$
such that $M-\Vert \y\Vert^2\geq0$ whenever $(\x,\y)\in\K$. Next, and possibly after re-scaling of the $g_j$'s,
we may and will set $M=1$, $\B=[-1,1]^n$. Next, let
\begin{eqnarray*}
\label{mom-lebesgue}
\gamma_\alpha\,:=\,&=&
\frac{1}{\lambda(\B)}\int_\B\x^\alpha\,d\lambda(\x),\qquad\alpha\in\N^n\\
&=&\left\{\begin{array}{l}\mbox{$0$ if $\alpha_i$ is odd for some $i$,}\\
\displaystyle\prod_{i=1}^n (\alpha_i+1)^{-1}\mbox{ otherwise}\end{array}\right.,
\end{eqnarray*}
be the moments of the (scaled) Lebesgue measure $\lambda$ on $\B$. In fact, as already mentioned,
one may consider any set $\B$ for which all moments of the Lebesgue measure (or any Borel measure with support exactly equal to $\B$)
are easy to compute (for instance an ellipsoid).

Moreover, letting  $g_{s+1}(\y):=1-\Vert \y\Vert^2$, and $x_i\mapsto \theta_i(\x):=1-x_i^2$, $i=1,\ldots,n$,
for convenience we redefine $\K\subset\R^n\times\R^m$ to be the basic semi-algebraic set
\begin{equation}
\label{new-K}
\K\,=\,\{\,(\x,\y)\::\:g_j(\x,\y)\,\geq\,0,\quad j=1,\ldots,s+1;\:\theta_i(\x)\geq0,\:i=1,\ldots,n\},
\end{equation}
and let $g_0\in\R[\x,\y]$ be the constant polynomial equal to $1$.
With $v_j:=\lceil{\rm deg}(g_j)/2\rceil$, $j=0,\ldots,m$, 
and for fixed $k\geq\max_{j}[v_j]$, consider the following optimization problem:
\begin{equation}
\label{sdp-k}
\begin{array}{rl}
\rho_k=\displaystyle\min_{p,\sigma_j,\psi_i}&\displaystyle\int_{\B} p(\x)\,d\lambda(\x)\\
\mbox{s.t.}&p(\x)-f(\x,\y)=\displaystyle\sum_{j=0}^{s+1}\sigma_j(\x,\y)\,g_j(\x,\y)+\sum_{i=1}^n \psi_i(\x,\y)\,\theta_i(\x)\\
&p\in\R[\x]_{2k};\:\sigma_j\in\Sigma_{k-v_j}[\x,\y],\:j=0,\ldots,s+1\\
&\psi_i\in\Sigma_{k-1}[\x,\y],\:i=1,\ldots,n.
\end{array}\end{equation}
For a feasible solution $p\in\R[\x]_{2k}$ of  (\ref{sdp-k}) the constraint certifies that $p(\x)-f(\x,\y)\geq0$ for all $(\x,\y)\in\K$,
and therefore, $p(\x)\geq\overline{J_f}(\x)$ for all $\x\in\B$. As minimizing $\displaystyle\int_\B p(\x)d\lambda(\x)$
is the same as minimizing $\displaystyle\int_\B (p(\x)-\overline{J_j}(\x))d\lambda(\x)$, by solving (\ref{sdp-k}) 
one tries to obtain a polynomial of degree at most $2k$ which dominates $\overline{J_f}$ on $\B$ 
and minimizes the $L_1$-nom $\displaystyle\int_\B (\vert p-\overline{J_j}\vert d\lambda$.
In other words, an optimal solution of (\ref{sdp-k}) is the best $L_1$-norm approximation in $\R[\x]_{2k}$
of $\overline{J_f}$ on $\B$ (from above).\\

It turns out that problem (\ref{sdp-k}) is a semidefinite program.
Indeed :

\noindent
- The criterion $\int_\B p(\x)\,d\lambda(\x)$ is linear in the coefficients $\p=(p_\alpha)$, $\alpha\in\N^n_{2k}$,
of the unknown polynomial $p\in\R[\x]_k$. In fact,
\[\int_\B p(\x)\,d\lambda(\x)\,=\,\sum_{\alpha\in\N^n_{2k}}p_\alpha\,\underbrace{\int_\B\x^\alpha\,d\lambda(\x)}_{\gamma_\alpha}\,=\,\sum_{\alpha\in\N^n_{2k}}p_\alpha\,\gamma_\alpha.\]
- The constraint
\[p(\x)-f(\x,\y)=\displaystyle\sum_{j=0}^{s+1}\sigma_j(\x,\y)\,g_j(\x,\y)+\sum_{i=1}^n \psi_i(\x,\y)\,\theta_i(\x),\]
with $p\in\R[\x]_{2k};\:\sigma_j\in\Sigma_{k-v_j}[\x,\y],\:j=0,\ldots,s$, and $\psi_i\in\Sigma_{k-1}[\x,\y],\:k=1,\ldots,n$,
reduces to 
\begin{itemize}
\item linear equality constraints between the coefficients of the polynomials $p,\sigma_j$ and $\psi_i$, to satisfy the identity, and
\item Linear Matrix Inequality (LMI) constraints to ensure that $\sigma_j$ and $\psi_i$ are all SOS polynomials of degree
bounded by $2(k-v_j)$ and $2(k-1)$ respectively.
\end{itemize}
The dual of the semidefinite program (\ref{sdp-k}) reads:
\begin{equation}
\label{sdp-k-dual}
\begin{array}{rl}
\rho_k^*=\displaystyle\min_{\z}&L_\z(f)\\
\mbox{s.t.}&\M_{k-v_j}(g_j\,\z)\,\succeq\,0,\quad j=0,\ldots,s+1\\
&\M_{k-1}(\theta_i\,\z)\,\succeq\,0,\quad i=1,\ldots,n\\
&L_\z(\x^\alpha)\,=\,\gamma_\alpha,\quad \alpha\in\N^n_{2k},
\end{array}\end{equation}
where $\z=(z_{\alpha\beta})$, $(\alpha,\beta)\in\N^{n+m}_{2k}$, and $L_\z:\R[\x,\y]\to\R$ is the Riesz functional introduced in \S \ref{notation}. 
Similarly, $\M_k(g_j\,\z)$ (resp. $\M_k(\theta_i\,\z)$) is the localizing matrix associated with the sequence
$\z$ and the polynomial $g_j$ (resp. $\theta_i$), also introduced in \S \ref{notation} (but now with $(\x,\y)$ instead of $\x$). 

Next we extend \cite[Theorem 3.5]{lass-param}  and prove that both (\ref{sdp-k}) and its dual
(\ref{sdp-k-dual}) have an optimal solution whenever $\K$ has nonempty interior.
\begin{theorem}
\label{th-sdp}
Let $\K$ be as in (\ref{new-K}) with nonempty interior,
and assume that $\K_\x\neq\emptyset$ for every $\x\in\B$. Then:

There is no duality gap between the semidefinite program (\ref{sdp-k}) and its dual (\ref{sdp-k-dual}).
Moreover (\ref{sdp-k}) (resp. (\ref{sdp-k-dual})) has an optimal solution $p_k^*\in\R[\x]_{2k}$
(resp. $\z^*=(z^*_{\alpha\beta})$, $(\alpha,\beta)\in\N^{n+m}_{2k}$), and
\begin{equation}
\label{confirm}
\lim_{k\to\infty}\:\int_\B \vert p^*_{k}(\x)-\overline{J_f}(\x)\vert\,d\x\,=\,0\quad\mbox{[Convergence in $L_1(\B)$]}.\end{equation}
\end{theorem}

\begin{proof}
As $\K$ has a nonempty interior it contains an open set $O\subset\R^n\times\R^m$.
Let $O_\x\subset\B$ be the projection of $O$ onto $\B$, so that its ($\R^n$) Lebesgue volume is positive.
Let $\mu$ be the finite Borel measure on $\K$ defined by
\[\mu(A\times B)\,:=\,\int_{A}\phi(B\,\vert\,\x)\,d\lambda(\x),\qquad A\in\mathcal{B}(\R^n),\:B\in\mathcal{B}(\R^m),\]
where for every $\x\in O_\x$, $\phi(d\y\,\vert\,\x)$ is the probability measure on $\R^m$, supported on $\K_\x$, and defined by:
\[\phi(B\,\vert\,\x)=\vol\,(\K_\x\cap B)/{\rm vol}(\K_\x),\qquad \forall B\in\mathcal{B}(\R^m).\]
On $\B\setminus O_\x$ the probability $\phi(d\y\,\vert\,\x)$ is an arbitrary probability measure on $\K_\x$.

Let $\z=(z_{\alpha\beta})$, $(\alpha,\beta)\in\N^{n+m}_{2k}$, be the moments of $\mu$.
As $\K\supset O$, $\M_{k-v_j}(g_j\,\z)\succ0$ (resp. $\M_{k-1}(\theta_i\,\z)\succ0$)
for $j=0,\ldots,s+1$ (resp. for $i=1,\ldots,n$). Indeed 
for all non zero  vectors $\u$ (indexed by the canonical basis $(\x^\alpha\y^\beta)$),
\begin{eqnarray*}
\langle\u,\M_{d-v_j}(g_j\,\z)\u\rangle&=&\int_\K u(\x,\y)^2g_j(\x,\y)\,d\mu(\x,\y)\\
&>&\int_O u(\x,\y)^2g_j(\x,\y)\,d\mu(\x,\y)\,>\,0\end{eqnarray*}
(as $g_j$ is supposed to be non trivial).
Moreover, by construction of $\mu$,
its marginal on $\B$ is the (scaled) Lebesgue measure $\lambda$ on $\B$ and so
\[L_\z(\x^\alpha)\,=\,\int_\B\x^\alpha\,d\lambda(\x)\,=\,\gamma_\alpha,\qquad\alpha\in\N^n_{2k}.\]
In other words, $\z$ is a strictly feasible solution of (\ref{sdp-k-dual}), i.e., Slater's condition holds for the semidefinite program 
(\ref{sdp-k-dual}). By a now standard result in convex optimization, this implies that $\rho_k=\rho^*_k$, and 
(\ref{sdp-k}) has an optimal solution if $\rho_k$ is finite. So it remains to show that indeed $\rho_k$ is finite
and (\ref{sdp-k-dual}) is solvable.

Observe that from the definition of the scaled Lebesgue measure $\lambda$ on $\B$,
$L_\z(1)=\gamma_0=1$. In addition, from the constraint $\M_{k-1}(g_{s+1}\,\z)\succeq0$, and $\M_{k-1}(\theta_i\,\z)\succeq0$, $i=1,\ldots,n$, we deduce that any feasible solution $\z$ of (\ref{sdp-k-dual}) satisfies:
\[L_\z(y_\ell^{2k})\,\leq\,1,\quad \forall\ell=1,\ldots,m;\qquad L_\z(x_i^{2k})\,\leq\,1,\quad\forall i=1,\ldots,n.\]
Hence by \cite[Proposition 3.6]{lass-book} this implies $\vert z_{\alpha\beta}\vert\leq \max[\gamma_0,1]=1$ for all $(\alpha,\beta)\in\N^{n+m}_{2k}$. Therefore, the feasible set is compact as closed and bounded, which in turn implies that
(\ref{sdp-k-dual}) has an optimal solution $\z^*$. And as  Slater's condition holds for (\ref{sdp-k-dual})
the dual (\ref{sdp-k}) also has an optimal solution.  Finally (\ref{confirm}) follows from
\cite[Theorem 3.5]{lass-param}. \hspace{1cm}$\qed$
\end{proof}

\begin{rem}
\label{monotone}
{\rm In fact, in Theorem \ref{robust-approx} one may impose the 
sequence $(p_k)\subset\R[\x]$, $k\in\N$, to be monotone, i.e.,
such that $\overline{J_f}\leq p_k\leq \overline{p_{k-1}}$ on $\B$, for all $k\geq2$. 
And similarly for Corollary \ref{design-approx}. For the practical computation of 
such a monotone sequence, in the semidefinite program (\ref{sdp-k})
it suffices to include the additional constraint (or positivity certificate)
\[p^*_{k-1}(\x)-p(\x)\,=\,\sum_{i=0}^n\phi_i(\x)\,\theta_i(\x),\quad\phi_0\in\Sigma[\x]_k,\,\phi_i\in\Sigma[\x]_{k-1},\,i\geq1,\]
where $\theta_0=1$
and $p^*_{k-1}\in\R[\x]_{k-1}$ is the optimal solution computed at the previous step $k-1$. In this case 
the inner approximations $(\bR_f^k)$, $k\in\N$, form a nested sequence since $\bR_f^{k}\subseteq\bR_f^{k+1}$ for all $k$.
Similarly the outer approximations $(\D_f^k)$, $k\in\N$, also form a nested sequence since $\D_f^{k+1}\subseteq\D_f^{k}$ for all $k$.
}\end{rem}

\section{Extensions}

\subsection{Semi-algebraic functions}
Suppose for instance that given $q_1,q_2\in\R[\x,\y]$, one wants to
characterize the set 
\[\{\,\x\in\B: \: \min\,[q_1(\x,\y),q_2(\x,\y)]\,\leq\,0\quad\mbox{for all $\y\in\K_\x$}\,\},\]
where $\K_\x$ has been defined in (\ref{setKx}), i.e., the set $\bR_f$
associated with the semi-algebraic function $(\x,\y)\mapsto f(\x,\y)=\min[q_1(\x,\y),q_2(\x,\y)]$.
If $f$ would be the semi-algebraic function $\max[q_1(\x,\y),q_2(\x,\y)]$, characterizing $\bR_f$ would
reduce to the polynomial case (with some easy adjustments) as $\bR_f=\bR_{q_1}\cap\bR_{q_2}$. But for $f=\min[q_1,q_2]$
this characterization is not so easy, and in fact is significantly more complicated.
However, even though $f$ is not a polynomial any more, 
we shall next see that the above methodology also works for semi-algebraic functions, a much larger class than the class of polynomials. Of course there is no free lunch and the resulting computational burden increases because one needs additional lifting variables to represent the semi-algebraic function.\\

With $\bS\subset\R^n$ being semi-algebraic, recall that a continuous function
$f:\bS\to\R$ is a semi-algebraic function if its graph $\{(\x,f(\x)):\x\in\bS\}$ is a semi-algebraic set. And in fact, the graph of every semi-algebraic function is the projection of some {\it basic} semi-algebraic set
in a lifted space. For more details the interested reader is referred to e.g. Lasserre and Putinar \cite[p. 418]{lass-put}.\\

So with $\K\subset\R^n\times\R^m$ as in (\ref{basic-K}), let $f:\K\to\R$ be a semi-algebraic function whose graph 
$\Psi_f=\{\,(\x,\y,f(\x,\y)):\,(\x,\y)\in\K\,\}$ is the projection $\{(\x,\y,v)\in\R^n\times\R^m\times\R\}$ of a basic semi-algebraic set
$\widehat{\K}\subset\R^n\times\R^m\times\R^r$, i.e.:
\[v\,=\,f(\x,\y)\mbox{ and }(\x,\y)\,\in\,\K\,\Longleftrightarrow\:\exists\,\w\mbox{ s.t. }(\x,\y,v,\w)\,\in\,\widehat{\K}.\]
Letting $\widehat{f}:\widehat{\K}\to\R$ be such that $\widehat{f}(\x,\y,v,\w):=v$, we have
\[\bR_f\,=\,\{\,\x\in\B\,:\,f(\x,\y)\,\leq\,0\mbox{ for all $\y$ such that  $(\x,\y)\in\K$}\,\}\]
\[=\{\,\x\in\B\,:\,\widehat{f}(\x,\y,v,\w))\,\leq\,0\mbox{ for all $(\y,v,\w)$ such that  $(\x,\y,v,\w))\in\widehat{\K}$}\,\}.\]
Hence this is just a special case of has been considered in \S \ref{main} and therefore converging
inner approximations of $\bR_f$ can be obtained as in Theorem \ref{robust-approx} and Theorem \ref{th-sdp}.

\begin{ex}
For instance suppose that $f:\R^n\times\R^m\to\R$ is the semi-algebraic function
$(\x,\y)\mapsto f(\x,\y):=\min[q_1(\x,\y),q_2(\x,\y)]$. Then
using $a\wedge b=\frac{1}{2}(a+b-\vert a-b\vert)$ and
$\vert a-b\vert =\theta\geq0$ with $\theta^2=(a-b)^2$,
\begin{eqnarray*}
\widehat{\K}\,=\,\left\{(\x,\y,v,w)\right.&:& (\x,\y)\in\K;\:w^2=(q_1(\x,\y)-q_2(\x,\y))^2; \quad w\geq0;\\
&&\left.2v=q_1(\x,\y)+q_2(\x,\y)-w\right\},\end{eqnarray*}
and
\[\Psi_f\,=\,\{\,(\x,\y,f(\x,\y)):(\x,\y)\in\K\,\}\,=\,\{(\x,\y,v):\:(\x,\y,v,w)\,\in\,\widehat{\K}\}.\]
\end{ex}

\subsection{Convex inner approximations} It is worth mentioning that
enforcing convexity of inner approximations of $\bR_f$ is easy.
But of course there is some additional computational cost and the convergence in
Theorem \ref{robust-approx} is lost in general.

To enforce convexity of the level set $\{\x\in\B: p^*_k(\x)\leq0\}$ it suffices to require
that $p^*_k$ is convex on $\B$, i.e., adding the constraint
\[\langle\u,\nabla^2 p^*_k(\x)\,\u\rangle\,\geq\,0,\qquad\forall (\x,\u)\in \B\times\U,\]
where $\U:=\{\u\in\R^n:\Vert \u\Vert^2\leq 1\}$. The latter constraint can in turn be 
enforced by the Putinar positivity certificate
\begin{equation}
\label{aux5}
\langle\u,\nabla^2 p^*_k(\x)\,\u\rangle\,=\,\sum_{i=0}^{n}\omega_i(\x,\u)\,\theta_i(\x)
+\omega_{n+1}(\x,\u)\,\theta_{n+1}(\x,\u),\end{equation}
for some SOS polynomials $(\omega_i)\subset\Sigma[\x,\u]$ (and where
$\theta_{n+1}(\x,\u)=1-\Vert\u\Vert^2$).

Then (\ref{aux5}) can be included in the semidefinite program (\ref{sdp-k})
with $\omega_0\in\Sigma[\x,\u]_{k}$, and $\omega_i\in\Sigma[\x,\u]_{k-1}$, $i=1,\ldots\,n+1$.
However, now $\z=(z_{\alpha,\gamma,\beta})$, $(\alpha,\beta,\gamma)\in\N^{2n+m}$, and so solving
the resulting semidefinite program is more demanding.

\subsection{Polynomial matrix inequalities}

Let $\A_\alpha\in\s_m$, $\alpha\in\N^n_d$, be real symmetric matrices and let $\B\subset\R^n$ be a given box.
We first consider the set 
\begin{equation}
\label{pmi-S}
\bS:=\{\,\x\in\B:\,\A(\x)\succeq0\,\},\end{equation}
where $\A\in\R[\x]^{m\times m}$ is the matrix polynomial
\[\x\mapsto\:\A(\x)\,:=\,\sum_{\alpha\in\N^n_d}\x^\alpha\,\A_\alpha.\]
If $\A(\x)$ is linear in $\x$ then $\bS$ is convex and (\ref{pmi-S})
is an LMI description of $\bS$ which is very nice as it can be used efficiently in semidefinite programming.

In the general case the description (\ref{pmi-S}) of $\bS$ is called a Polynomial Matrix Inequality (PMI)
and cannot be used as efficiently as in the convex case. Indeed
$\bS$ is a basic semi-algebraic set with an alternative description in terms of
the box constraint $\x\in\B$ and $m$ additional polynomial inequality constraints
(including the constraint  ${\rm det}(\A(\x))\geq0$).
However, this latter description may not be very appropriate either because the degree
of polynomials involved in that description is potentially as large as $d^m$
which precludes from its use for practical computation (e.g., for optimization purposes). 

On the other hand,
for polynomial optimization problems with a PMI constraint $\A(\x)\succeq0$, one may still define
an appropriate and {\it ad hoc} hierarchy of semidefinite relaxations, as described in
Hol and Scherer \cite{hol1,hol2}, and Henrion and Lasserre \cite{ieee-pmi}. But even if more economical than the hierarchy using the former description
of $\bS$ with $m$ (high degree) polynomials, this latter approach may not still be ideal. In particular
it is not clear how to detect (and then take benefit of) some possible structured sparsity to reduce the computational cost.

So in the general case and when $d^m$ is not small,
one may be interested in a  description of $\bS$
simpler than the PMI (\ref{pmi-S}) so that it can 
used more efficiently.\\

Let $\Y:=\{\,\y\in\R^m: \Vert\y\Vert^2=1\,\}$ denote the unit sphere of $\R^m$.
Then with $(\x,\y)\mapsto f(\x,\y):=-\langle\y,\A(\x)\y\rangle$, the set $\bS$ has the alternative and equivalent description 
\begin{equation}
\label{pmi-S-1}
\bS\,=\,\{\,\x\in\B:\,f(\x,\y)\,\leq\,0,\:\forall\y\in\Y\,\}\,=:\,\bR_f,
\end{equation}
which involves the quantifier ``$\forall$". 
Therefore the machinery developed in \S \ref{main} can be applied to define 
the hierarchy of inner approximations $\bR_f^k\subset\bS$ in Theorem \ref{robust-approx}, where for each $k$,
$\bR^k_f=\{\x\in\B: p_k(\x)\leq0\}$ for some polynomial $p_k$ of degree $k$.
Observe that if $\x\mapsto \A(\x)$ is not a constant matrix, then with
\[\x\mapsto \overline{J_f}(\x)\,:=\,\sup_\y\,\{\,f(\x,\y):\,\y\in\Y\,\},\qquad\x\in\B,\]
the set $\{\x:\overline{J_f}(\x)=0\}$ has Lebesgue measure zero because
$\overline{J_f}(\x)$ is the largest eigenvalue of $-\A(\x)$.
Hence by Theorem \ref{robust-approx}
\[{\rm vol}\left(\bR_f^k\right)\,\to\,{\rm vol}(\bS),\qquad\mbox{ as $k\to\infty$}.\]
Notice that computing $p_k$ has required to introduce the $m$ additional variables $\y$ but
the degree of $f$ is not larger than $d+2$ if $d$ is the maximum degree of the entries.\\

Importantly for computational purposes, structure sparsity can be exploited to reduce the computational burden.
Write the polynomial 
\[(\x,\y)\mapsto f(\x,\y)\,=\,-\langle \y,\A(\x)\,\y\rangle\,=\,\sum_{\alpha\in\N^n}h_\alpha(\y)\,\x^\alpha,\quad (\x,\y)\in\R^n\times\R^m,\]
for finitely many quadratic polynomials $\{h_\alpha\in\R[\y]_2:\alpha\in\N^m\}$.
Suppose that the polynomial 
\[\x\mapsto \theta(\x)\,:=\,\sum_{\alpha\in\N^n}h_\alpha(\y)\,\x^\alpha,\qquad \x\in\R^n,\]
has some structured sparsity. That is, $\{1,\ldots,n\}=\cup_{\ell=1}^s I_\ell$ (with possible overlaps)
and $\theta(\x)=\sum_{\ell=1}^s\theta_\ell(\x_\ell)$ where $\x_\ell=\{x_i: i\in I_\ell\}$.
Then $(\x,\y)\mapsto f(\x,\y)$ inherits the same structured sparsity (but with now 
$\cup_{\ell=1}^sI'_\ell$ where $I'_\ell=\{x_i,y_1,\ldots y_m:i\in I_\ell\}$).  And so in particular, for computing
$p_k$ one may use the sparse version of the hierarchy of semidefinite relaxations 
introduced in  Waki et al. \cite{waki} which permits to handle problems with 
a significantly large number of variables.
\begin{ex}
\label{chap8-ex1}
{\rm The following illustrative example is taken from Henrion and Lasserre \cite{lass-ieee}.
With $n=2$, let $\B\subset\R^2$ be the unit disk $\{\x\,:\,\Vert\x\Vert^2\leq 1\}$, and let
\[\A(\x)\,:=\,\left[\begin{array}{cc} 1-16x_1x_2 & x_1 \\
x_1 & 1-x_1^2-x_2^2 \end{array}\right] ;\quad \bS:=\{\x\in\B\::\:\A(\x)\,\succeq\,0\}.\]

\begin{figure}[ht]
\begin{center}
\resizebox{0.9\textwidth}{!}
{\includegraphics{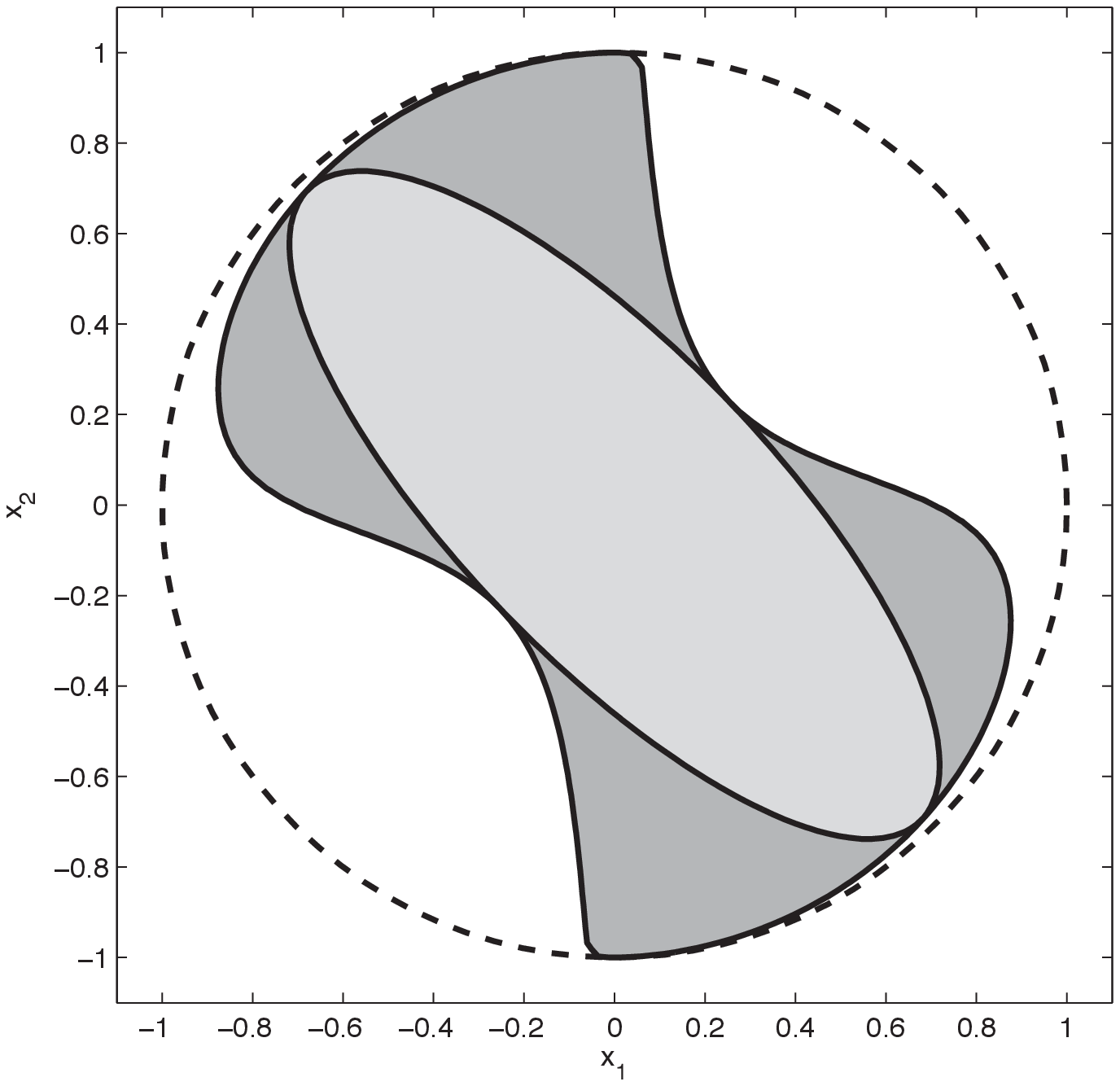}\includegraphics{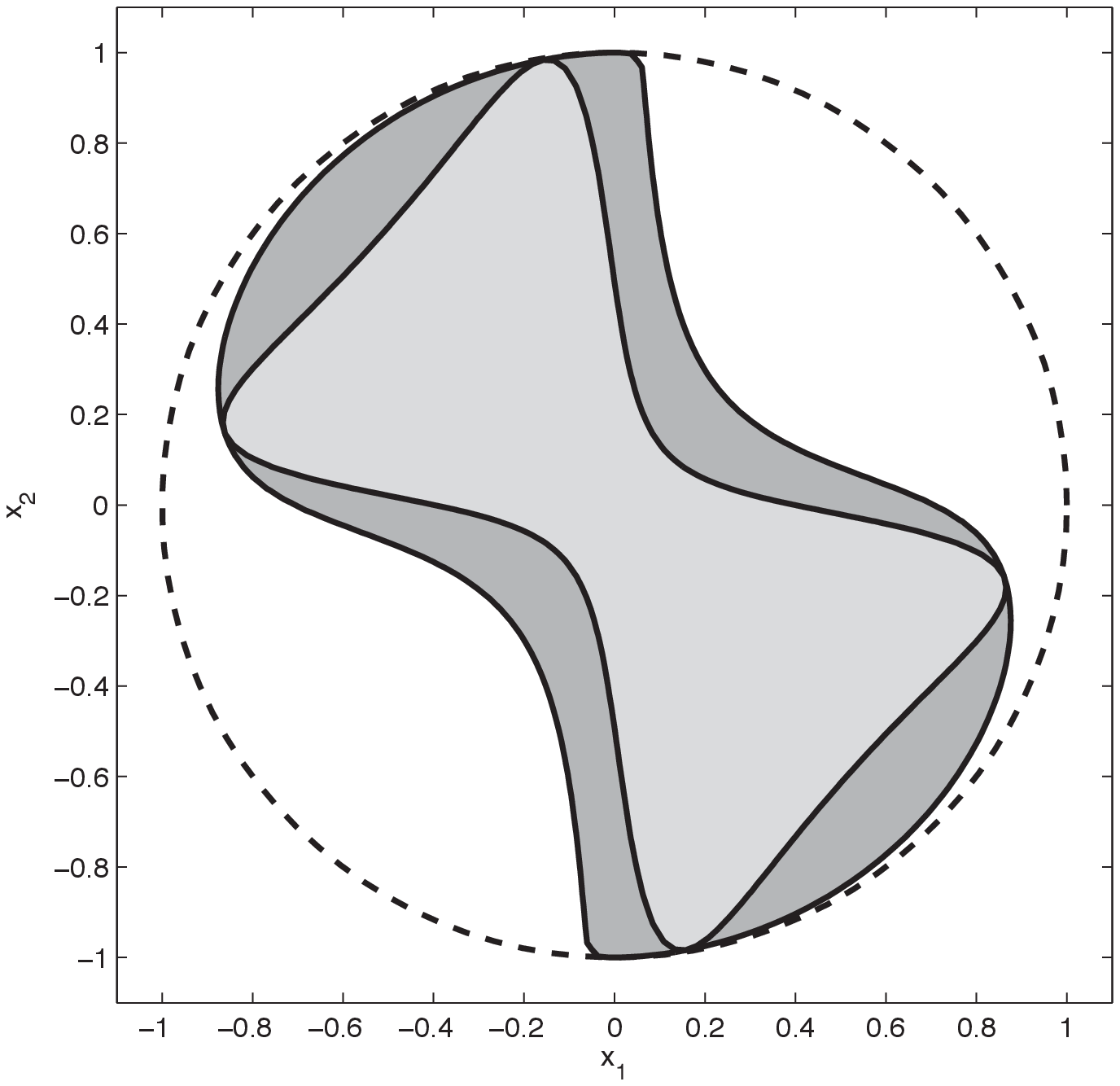}}
\caption{Example \ref{chap8-ex1}: $\bR_f^1$ (left) and $\bR_f^2$ (right) inner approximations (light gray)
of $\bS$ (dark gray) embedded in unit disk $\B$ (dashed)\label{figqmidisk24}}
\end{center}
\end{figure}

\begin{figure}[ht]
\begin{center}
\resizebox{0.9\textwidth}{!}
{\includegraphics{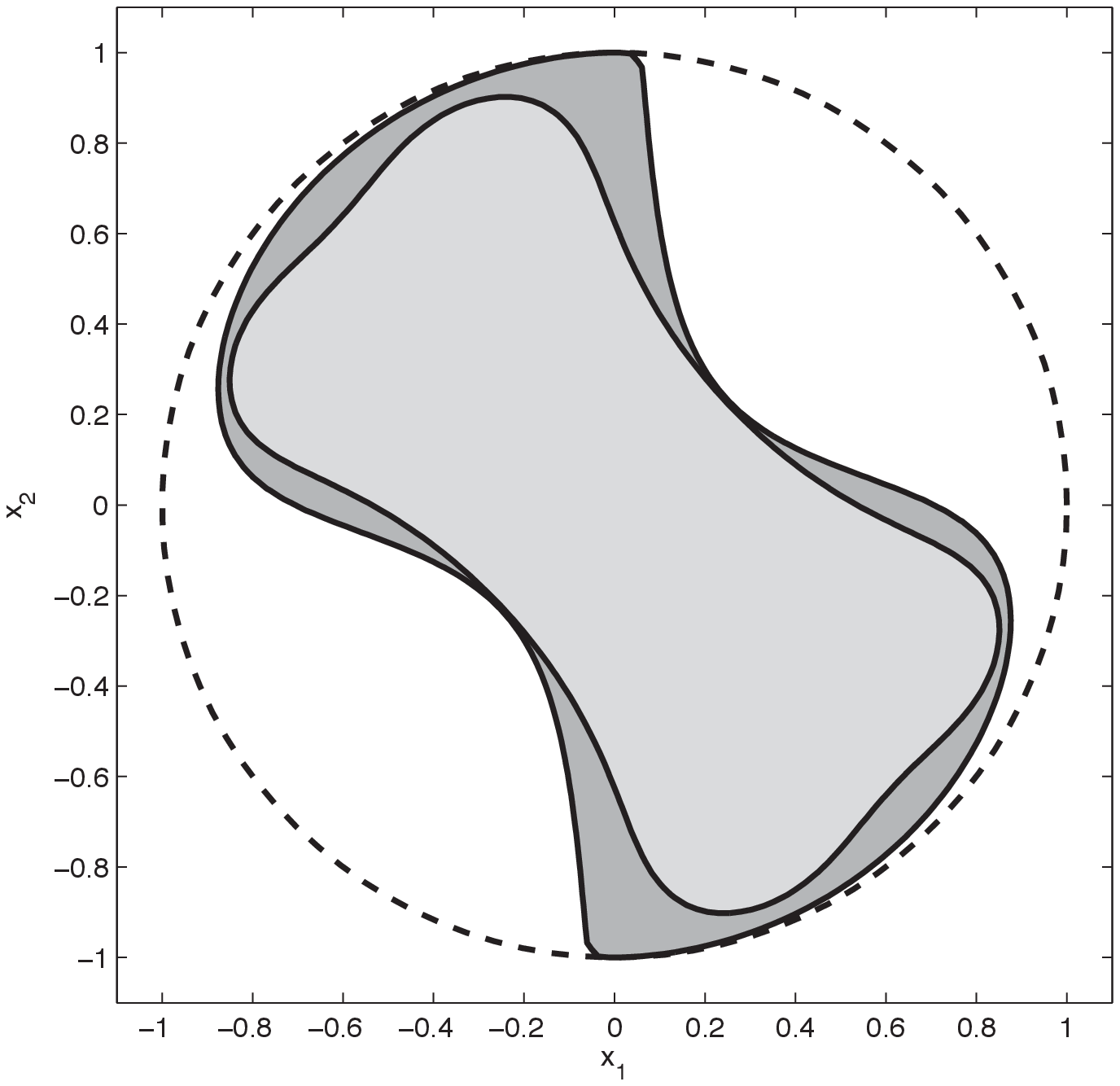}\includegraphics{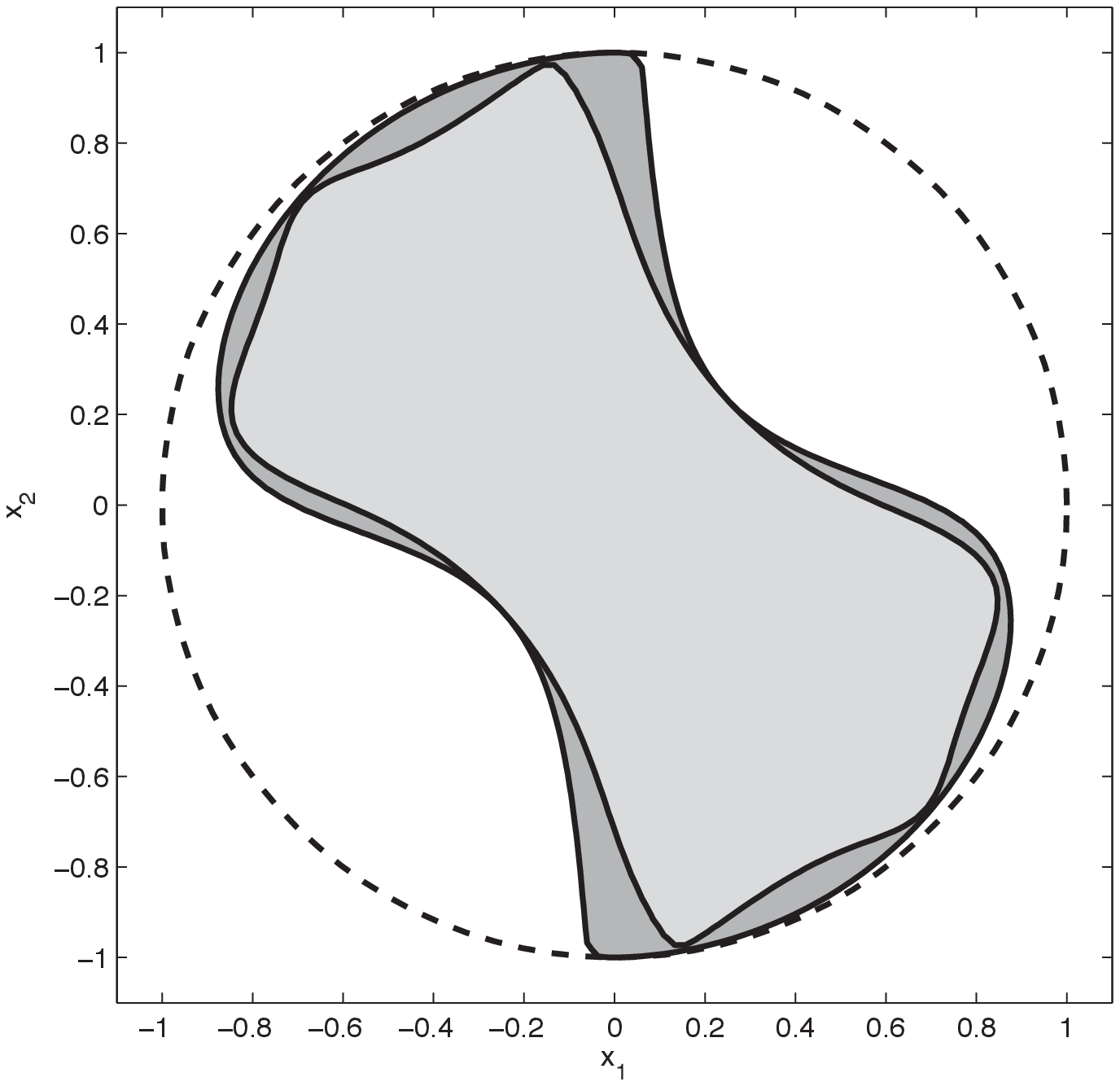}}
\caption{Example \ref{chap8-ex1}: $\bR_f^3$ (left) and $\bR_f^4$ (right) inner approximations (light gray)
of $\bS$ (dark gray) embedded in unit disk $\B$ (dashed).\label{figqmidisk68}}
\end{center}
\end{figure}
In Figure \ref{figqmidisk24} is displayed $\bS$
and the degree two $\bR_f^1$ and
four $\bR_f^2$ inner approximations of $\bS$, whereas in Figure \ref{figqmidisk68}
are displayed the $\bR_f^3$ and
$\bR_f^4$ inner approximations of $\bS$. One may see that with
$k=4$, $\bR_f^4$ is already a quite good approximation of $\bS$.
}\end{ex}
Next, with $\K$ as in (\ref{basic-K}) consider now sets of the form
\[\bR_f:=\{\x\in\B: \bF(\x,\y)\preceq0\mbox{ for all $\y$ in $\K_\x$}\,\},\]
where $\bF\in\R[\x,\y]^{m\times m}$ is a polynomial matrix in the $\x$ and $\y$ variables. Then 
letting $\Z:=\{\z\in\R^m:\Vert \z\Vert=1\}$, $\widehat{\K}:=\K\times\Z$, and $f:\widehat{\K}\to\R$, defined by:
\[(\x,\y,\z)\mapsto f(\x,\y,\z)\,:=\,\langle \z,\bF(\x,\y)\,\z\rangle,\qquad (\x,\y,\z)\in\widehat{\K},\] 
the set $\bR_f$ has the equivalent description:
\[\bR_f:=\{\x\in\B: f(\x,\y,\z)\leq 0\mbox{ for all $(\y,\z)$ in $\widehat{\K}_\x$}\,\},\]
and the methodology of \S \ref{main} again applies.

\subsection{Several functions $f$}

We next consider sets of the form
\[\bR_F:=\{\,\x\in\B:\, \mbox{$(\x,\y)\in \mathbf{F}$ for all $\y$ such that $(\x,\y)\in\K$}\,\}\]
where $\mathbf{F}\subset\R^n\times\R^m$ is a basic-semi algebraic set defined by
\[\mathbf{F}\,:=\,\{\,(\x,\y)\in\R^n\times\R^m:\,f_\ell(\x,\y)\,\leq\,0,\quad\forall \ell=1,\ldots,q\,\},\]
for some polynomials $(f_\ell)\subset\R[\x,\y]$, $\ell=1,\ldots,q$. In other words,
\[\bR_f\,=\,\{\,\x\in\B:\: \K_\x\,\subset\,\mathbf{F}_\x\,\},\]
where $\mathbf{F}_\x:=\{\,\y:\:(\x,\y)\,\in\,\mathbf{F}\,\}$.

Of course it is a particular case
of the previous section with the semi-algebraic function $f=\max[f_1,\ldots,f_q]$,
but in this case a simpler approach is possible. Let $p_{k\ell}\in\R[\x]$ be 
an optimal solution of (\ref{sdp-k}) associated with $f_\ell$,
$\ell=1,\ldots,q$, and let the set $\bR_F^k$ be defined by
\[\bR_F^k\,:=\,\{\,\x\in\R^n:p_{k\ell}(\x)\,\leq\,0,\quad \ell=1,\ldots,q\,\}\,=\,\bigcap_{\ell=1}^q\,\bR^k_{f_\ell},\]
where for each $\ell=1,\ldots,q$, the set $\bR_{f_\ell}^k$ is defined in the obvious manner.

The sets $(\bR_F^k)\subset\bR_F$, $k\in\N$, provide a sequence of inner approximations of $\bR_F$ with the nice property that
\[\lim_{k\to\infty}\,{\rm vol}\left(\bR_F^k\right)\,=\,{\rm vol}\left(\bR_F\right),\]
whenever the set $\{\,\x\in\B:\max_{\ell}\overline{J_{f_\ell}}(\x)\,=0\,\}$
has Lebesgue measure zero.

\subsection{Sets defined with two quantifiers}

Consider three types of variables $(\x,\y,\u)\in\R^n\times\R^m\times\R^s$,
a box $\B_\x\subset\R^n$, a box $\B_\y\subset\R^m$,
and a compact basic semi-algebraic set $\K\subset\B_\x\times\B_\y\times\U$.
It is assumed that for each $(\x,\y)\in\B_{\x\y}\,(=\B_\x\times\B_\y)$,
\[\K_{\x\y}\,:=\,\{\u\in\U\::\:(\x,\y,\u)\,\in\,\K\}\,\neq\,\emptyset.\]

\noindent
{\bf Sets with $\mathbf{\exists}$, $\mathbf{\forall}$.}
Consider a set $\D_f'$ of the form
\begin{equation}
\label{dprime}
\D_f'\,:=\,\{\x\in\B_\x\::\: \mbox{$\exists\,\y\in\B_\y$ such that $f(\x,\y,\u)\leq0$ for all $\u\in\K_{\x\y}$}\}.\end{equation}
Such a set is not easy to handle, in particular for optimizing over it.
So it is highly desirable to approximate as closely as possible the set $\D'_f$ with
a set having a much simpler description, and in particular a description with {\it no} quantifier. We propose to
use the methodology of \S \ref{main} to provide such approximations.
Consider the lift of $\D'_f$ given by:
\begin{eqnarray*}
\H_f&:=&\{\,(\x,\y)\in\B_{\x\y}:f(\x,\y)\,\leq\,0\mbox{ for all $\u\in\K_{\x\y}$}\,\}\\
&=&\{\,(\x,\y)\in\B_{\x\y}:\overline{J_f}(\x,\y)\leq 0\,\},
\end{eqnarray*}
where $\overline{J_f}(\x,\y):=\max_\u\,\{\,f(\x,\y,\u):(\x,\y,\u)\in\K\,\}$. Using results of \S \ref{main} we can find inner 
approximations $(\H^k_f)$ of the form:
\[\H^k_f\,:=\,\{\,(\x,\y)\in\B_{\x\y}:p_k(\x,\y)\leq 0\}\,\subset\,\H_f,\quad k\in\N,\]
for some polynomials $(p_k)\subset\R[\x,\y]$. This then gives inner approximations $(\D^k_f)$ of $\D'_f$ of the form:
\[\D^k_f\,=\,\{\,\x\in\B_{\x}:p_k(\x,\y)\,\leq\,0\mbox{ for some $\y\in\B_\y$}\,\}\,\subset\,\D'_f,\quad k\in\N.\]
Considering again results of \S \ref{main} we can find outer approximations
of these inner approximations, of the form:
\[\D^{k\ell}_f\,=\,\{\,\x\in\B_{\x}:p_{k\ell}(\x)\,\leq\,0\,\}\,\supset\,\D^{k}_f,\quad k,\ell\in\N,\]
for some polynomials $(p_{k\ell})\subset\R[\x]$. Unfortunately
obtaining (some type of) convergence 
$\D_f^{k\ell}\to\D'_f$ is much more difficult and requires additional hypotheses.
\vspace{0.2cm}

\noindent
{\bf Sets with $\forall$, $\exists$.} Consider now a set $\bR'_f$ of the form
\begin{equation}
\label{rprime}
\bR'_f\,:=\,\{\x\in\B_\x\::\: \mbox{$\forall\,\y\in\B_\y,\,\exists\,\u\in\K_{\x\y}$ such that $f(\x,\y,\u)\geq0$}\}.\end{equation}
As for $\D'_f$, such a set is not easy to handle, in particular for optimizing over it.
So again it is highly desirable to approximate as closely as possible the set $\bR'_f$ with
a set having a much simpler description, and in particular a description with {\it no} quantifier. So
proceeding in a similar fashion as before,
\[\bR'_f=\{\,\x\in\B_\x: \mbox{ $\overline{J_f}(\x,\y)\geq0$ for all $\y\in\B_\y$}\, \},\]
and from \S \ref{main} we can provide outer approximations $(\bR^k_f)$ of $\bR'_f$ of the form:
\[\bR'_f\subset\bR_f^k\,=\{\,\x\in\B_\x: \mbox{$p_k(\x,\y)\geq0$ for all $\y\in\B_\y$}\},\quad k\in\N,\]
for some polynomials $(p_k)\subset\R[\x,\y]$. 
Considering again results of \S \ref{main} we can find inner approximations
of these outer approximations, of the form:
\[\bR^{k\ell}_f\,=\,\{\,\x\in\B_{\x}:p_{k\ell}(\x)\,\geq\,0\,\}\,\subset\,\bR^k_f,\quad k,\ell\in\N,\]
for some polynomials $(p_{k\ell})\subset\R[\x]$. Unfortunately,  as for approximating $\D'_f$,
obtaining (some type of) convergence 
$\bR^{k\ell}_f\to\bR'_f$ is also much more difficult and requires additional hypotheses.

\subsection{General semi-algebraic set}

We finally briefly discuss the case where the set $\K\subset\R^{n+m}$ in (\ref{basic-K}) is semi-algebraic 
but not {\it basic} semi-algebraic. Of course $\K$ can be always represented as the projection of some basic semi-algebraic set 
$\widehat{\K}$ defined in a higher dimensional space $\R^{n+m+t}$ for some $t\in\N$. That is
\[\K\,=\,\{(\x,\y):\:\exists \,\z\mbox{ such that }(\x,\y,\z)\in\widehat{\K}\,\}.\]
If $\widehat{\K}$ is known (i.e. $\K$ is known implicitly from $\widehat{\K}$) then 
sets of the form $\D_f$ in (\ref{design-set}) can be approximated as we have done in \S \ref{main}.
Indeed,
\begin{eqnarray*}
\D_f&=&\{\:\x\in\B\::\:f(\x,\y)\geq 0 \mbox{ for some $\y\in\K_\x$}\:\}\\
&=&\{\:\x\in\B\::\:f(\x,\y)\geq 0 \mbox{ for some $(\y,\z)\in\widehat{\K}_\x$}\:\},
\end{eqnarray*}
where for every $\x\in\B$, $\widehat{\K}_\x:=\{\,(\y,\z):(\x,\y,\z)\in\widehat{\K}\,\}$. Similarly,
sets of the form $\bR_f$ in (\ref{robust-set}) can be also approximated as we have done in \S \ref{main} because
\begin{eqnarray*}
\bR_f&=&\{\:\x\in\B\::\:f(\x,\y)\leq 0 \mbox{ for all $\y\in\K_\x$}\:\}\\
&=&\{\:\x\in\B\::\:f(\x,\y)\leq 0 \mbox{ for all $(\y,\z)\in\widehat{\K}_\x$}\:\}.
\end{eqnarray*}
\noindent
{\bf Finite union of basic semi-algebraic sets.} Alternatively, a general semi-algebraic set $\K$ is often
described as the finite union $\cup_t\K_t$ (with possible overlaps) of basic compact semi-algebraic sets $\K_t$, $t\in T$, defined by:
\[\K_t\,=\,\{\,(\x,\y)\in\R^{n+m}:\:g_{tj}(\x,\y)\,\geq\,0,\quad j=1,\ldots,s_t\,\},\quad t\in T,\]
for some polynomials $(g_{tj})\subset\R[\x,\y]$, $j=1,\ldots,s_t$, $t\in T$.
Again the function 
\[\x\mapsto \overline{J_f}(\x)\,:=\,\max_\y\,\{\,f(\x,\y):(\x,\y)\,\in\,\K\,\},\quad \x\in\B,\]
is upper-semicontinuous on $\B$ and so Corollary \ref{over-under} applies. 

Therefore we can apply again the methodology of \S \ref{main} with {\it ad hoc} adjustments.
In particular, for practical computation of a sequence of polynomials $(p_k)\subset\R[\x]$ of increasing degree and with 
$p_k\geq\overline{J_f}$ for all $k$, and  $\int_\B\vert p_k-\overline{J_f}\vert\,d\lambda\to 0$ as $k\to\infty$, 
the analogue of the semidefinite program (\ref{sdp-k}) reads:
\begin{equation}
\label{sdp-union-k}
\begin{array}{rl}
\rho_k=\displaystyle\min_{p,\sigma_{tj},\psi_{ti}}&\displaystyle\int_{\B} p(\x)\,d\lambda(\x)\\
\mbox{s.t.}&p-f=\displaystyle\sum_{j=0}^{s_t+1}\sigma_{tj}\,g_{tj}+\sum_{i=1}^n \psi_{ti}\,\theta_i,\quad t\in T\\
&p\in\R[\x]_{2k};\:\sigma_{tj}\in\Sigma_{k-v_j}[\x,\y],\:j=0,\ldots,s_t+1;\,t\in T\\
&\psi_{ti}\in\Sigma_{k-1}[\x,\y],\:i=1,\ldots,n;\,t\in T.
\end{array}\end{equation}
If ${\rm int}\,(\K_t)\neq\emptyset$ for every $t\in T$ and $\K_\x:=\{\y:(\x,\y)\in\K\,\}\neq\emptyset$ for every $\x\in\B$,  then the analogue of Theorem \ref{th-sdp} is valid.
Namely, the semidefinite program (\ref{sdp-union-k}) has an optimal solution
$p^*_k\in\R[\x]_{2k}$ with the desired convergence property:
\[\lim_{k\to\infty}\:\int_\B \vert p^*_{k}(\x)-\overline{J_f}(\x)\vert\,d\x\,=\,0\quad\mbox{[Convergence in $L_1(\B)$]}.\]

\section{Conclusion}

We have seen how to approximate some semi-algebraic sets defined with quantifiers by a monotone sequence
of sublevel sets associated with appropriate polynomials of increasing degree. Each polynomial of the sequence is computed by solving
a semidefinite program whose size increases with the degree of the polynomial and
convergence of the approximations takes place in a strong sense.
Several extensions have also been provided. Of course,
solving the resulting hierarchy of semidefinite programs is computationally expensive 
and so far, in its present form the methodology is limited to problems of modest size only. Fortunately,
larger size problems with some structured sparsity pattern can be attacked 
by applying techniques already used in \cite{waki}. However,
evaluating (and improving) the efficiency of this methodology on a sample of problems of significant size
is a topic of further investigation.

\end{document}